\documentclass{amsart}  
\usepackage{amsthm, amsmath, amscd, amssymb, latexsym, stmaryrd, color}
\usepackage[top=3.0cm, bottom=3.0cm, left=3.0cm, right=3.0cm]{geometry}
\theoremstyle{plain}

\newtheorem{theorem}{Theorem}[section]
\newtheorem{lemma}[theorem]{Lemma}
\newtheorem{proposition}[theorem]{Proposition}
\newtheorem{prop-def}[theorem]{Proposition-Definition}
\newtheorem{corollary}[theorem]{Corollary}

\theoremstyle{definition}

\newtheorem{remark}[theorem]{Remark}
\theoremstyle{remark}
\newtheorem*{ack}{Acknowledgement}
\usepackage[mathscr]{eucal}
\usepackage{graphics, graphpap}
\usepackage{array, tabularx, longtable}
\usepackage{color}
\numberwithin{equation}{section}

\def\Pic{\mathrm{Pic}}
\def\Hom{\mathrm{Hom}}
\def\Spec{\mathrm{Spec}}
\def\Int{\mathrm{Int}}

\def\coker{\mathrm{coker}}

\def\Ex{\mathrm{Ex}}
\def\Supp{\mathrm{Supp}}
\def\Jac{\mathrm{Jac}}

\def\lct{\mathrm{lct}}
\def\red{\mathrm{red}}
\def\gcd{\mathrm{gcd}}
\def\Sp{\mathrm{Sp}}
\def\Id{\mathrm{Id}}
\def\HS{\mathrm{MHS}}
\def\hsp{\mathrm{Hsp}}
\def\mon{\mathrm{mon}}
\def\Sp{\mathrm{Sp}}

\def\p{\mathbf{p}}

\title[Homogeneous singularities]{\bf On complex homogeneous singularities}  

\author[L\^e]{L\^e Quy Thuong}
\address{Department of Mathematics, Vietnam National University \newline\indent 334 Nguyen Trai Street, Thanh Xuan District, Hanoi, Vietnam}
\email{leqthuong@gmail.com}
\address{BCAM - Basque Center for Applied Mathematics \newline \indent Alameda de Mazarredo 14, E-48009 Bilbao, Basque Country, Spain}
\email{qle@bcamath.org}

\author[Nguyen]{Nguyen Phu Hoang Lan}
\address{Department of Mathematics, Vietnam National University \newline\indent 334 Nguyen Trai Street, Thanh Xuan District, Hanoi, Vietnam}
\email{nphlan@gmail.com}

\author[Pho]{Pho Duc Tai}
\address{Department of Mathematics, Vietnam National University \newline\indent 334 Nguyen Trai Street, Thanh Xuan District, Hanoi, Vietnam}
\email{phoductai@gmail.com}

\thanks{This research is funded by the Vietnam National Foundation for Science and Technology Development (NAFOSTED) under grant number FWO.101.2015.02.}

\thanks{The first author's research is also supported by ERCEA Consolidator Grant 615655 - NMST and by the Basque Government through the BERC 2014-2017 program and by Spanish Ministry of Economy and Competitiveness MINECO: BCAM Severo Ochoa excellence accreditation SEV-2013-0323.} 

\keywords{Homogeneous singularity, Log-resolution, Local systems, Multiplier ideals, Finite abelian covers, Hodge spectrum, spectrum multiplicity, Monodromy zeta function}
\subjclass[2010]{Primary 14B05, 14C20, 14F18, 32S20}

\begin{document}   

\begin{abstract}
In this article, we consider the singularity of an arbitrary homogeneous polynomial with complex coefficients $f(x_0,\dots,x_n)$ at the origin of $\mathbb C^{n+1}$, via the study of the monodromy characteristic polynomials $\Delta_l(t)$, and the relation between the monodromy zeta function and the Hodge spectrum of the singularity. We go further with $\Delta_1(t)$ in the case $n=2$. This work is based on knowledge of multiplier ideals and local systems.
\end{abstract}
\maketitle                 

\section{Introduction}
Let $f$ be a homogeneous polynomial (not necessarily reduced) of degree $d$ in $n+1$ variables with coefficients in $\mathbb C$, which defines a holomorphic function germ at the origin $O$ of $\mathbb C^{n+1}$. In general, according to \cite{Mil} and \cite{Le}, the Milnor fiber of $(f,O)$ is up to diffeomorphism a manifold $M=f^{-1}(\delta)\cap B_{\varepsilon}$, for $B_{\varepsilon}\subset \mathbb C^{n+1}$ a ball of radius $\varepsilon$ around $O$ and $0<\delta \ll \varepsilon \ll 1$, which has the homotopy type of a bouquet of $\mu$ spheres of dimension $n$. Since here $f$ is a homogeneous polynomial, however, $f^{-1}(\delta)\cap B_{\varepsilon}$ is a deformation retract of $f^{-1}(\delta)\cong f^{-1}(1)$, thus we may consider $M$ as $f^{-1}(1)$. The monodromy $T:H^*(M,\mathbb C)\to H^*(M,\mathbb C)$ of the singularity may be given explicitly to be the $\mathbb C$-linear endomorphism induced by the map $(x_0,\dots,x_n)\mapsto (e^{\frac{2\pi i}{d}}x_0,\dots,e^{\frac{2\pi i}{d}}x_n)$. It becomes classical for $f$ being an isolated homogeneous singularity at $O$ where many important invariants such as the Milnor number $\mu$, the characteristic polynomials of $T$, the signature and Hodge numbers of $M$ are computed completely in topological and algebraic methods as well as via mixed Hodge structures (cf. \cite{MO}, \cite{St1}). 

In the case where $f$ is a reduced homogeneous polynomial, Esnault \cite{Esn} introduced a method to compute the Betti numbers, the rank and the signature of the intersection matrices of the singularity $(f,O)$, using mixed Hodge structures on cohomology groups of the Milnor fiber $M$ and the existence of spectral sequences converging to the cohomology groups, together with resolution of singularity. The work by Esnault definitely inspired the study by Loeser-Vaqui\'e \cite{LV} of the Alexander polynomial of a reduced complex projective plane curve, where they provided a formula for the Alexander polynomial of such a curve which generalizes the previous one by Libgober \cite{Lib1, Lib2}. It is likely that the approaches of Libgober in \cite{Lib2} and Loeser-Vaqui\'e in \cite{LV}, as well as the work by Nadel in \cite{Nad}, are also starting points of the studies on multiplier ideals and local systems, which were thereafter studied strongly by Esnault-Viehweg \cite{EV}, Ein-Lazarsfeld \cite{EL}, Demailly \cite{Dm}, Kollar \cite{Ko97}, Budur \cite{Bu1, Bu3}, Budur-Saito \cite{BS1}. 

Due to the development of the theory of multiplier ideals and local systems, Budur \cite{Bu4} gives an explicit description of the local system of the complement in $\mathbb P^n$ of the divisor defined by a homogeneous polynomial $f$ {\it without the condition of reducedness}. In the present work, we use Budur's article \cite{Bu4} to study the characteristic polynomials, the Hodge spectrum and the monodromy zeta function of an arbitrary homogeneous hypersurface singularity. Let us now review in a few words what we shall do. We denote by $D$ the closed subscheme of $\mathbb P^n$ defined by the zero locus of a degree $d$ homogeneous polynomial $f$ and by $U$ the complement of $D$ in $\mathbb P^n$. Then, as shown in \cite{Bu3, Bu4}, there is an eigensheaf decomposition of the $\mathcal O_U$-module sheaf $\sigma_*\mathbb C_M$ into the unitary local systems $\mathcal V_k$ on $U$ given by the eigensheaf of $T$ with respect to the eigenvalue $e^{-\frac{2\pi ik} d}$, $0\leq k\leq d-1$, where $\sigma$ is the canonical projection $M\to U$. On cohomology level, using the Leray spectral sequence, one gets $H^l(U,\mathcal V_k)$ to be the eigenspace of the monodromy $T$ on $H^l(M,\mathbb C)$ with respect to the eigenvalue $e^{-\frac{2\pi ik} d}$, for any $l$ in $\mathbb N$ (cf. \cite{Bu4}). Assume that $D$ has $r$ distinct irreducible components $D_i$ and that $m_i$ is the multiplicity of $D_i$ in $D$. By \cite[Lemma 4.2]{Bu4}, for each $k$, modulo the identification $RH$ in \cite[Theorem 1.2]{Bu3}, the local system $\mathcal V_k$ is nothing but the element $(\mathcal O_{\mathbb P^n}(\sum_{j=1}^r\{\frac{km_j}{d}\}d_j), (\{\frac{km_1}{d}\},\dots,\{\frac{km_r}{d}\}))$  in the group $\Pic^{\tau}(\mathbb P^n,D)$ of {\it realizations of boundaries of $\mathbb P^n$ on $D$} (cf. \cite[Definition 1.1]{Bu3}). 

The problem of computing the complex dimension of $H^l(U,\mathcal V_k)$ can be solved completely under the works by Budur \cite{Bu1, Bu2, Bu3, Bu4} in terms of resolution of singularity. Let $\pi: Y\to \mathbb P^n$ be a log-resolution of the family $\{D_1,\dots, D_r\}$, with exceptional divisor $E=\pi^*(\bigcup_{j=1}^rD_j)=\sum_{j\in A}N_jE_j$, $E_j$ being irreducible. Denote by $\mathscr L^{(k)}$ the invertible sheaf $\pi^*\mathcal O_{\mathbb P^n}\left(\sum_{j=1}^r\left\{\frac{km_j}{d}\right\}d_j\right) \otimes \mathcal O_Y\left(-\left\lfloor{\sum_{j=1}^r\left\{\frac{km_j}{d}\right\}}\pi^*D_j\right\rfloor\right)$ on $Y$. As proved in Lemma \ref{dimHi}, we get
$$\dim_{\mathbb C}H^l(U,\mathcal V_{d-k})=\sum_{p\geq 0}\dim_{\mathbb C}H^{l-p}(Y,\Omega_Y^p(\log E)\otimes {\mathscr L^{(k)}}^{-1}),$$
for $l\geq 0$ and $1\leq k\leq d$, from which the characteristic polynomial $\Delta_l(t)$ of $T$ on $H^l(M,\mathbb C)$ follows. Observe that this description is not really useful in practice since it is too difficult to compute the number on the right hand side of the previous equality. However, in the special case where $n=2$ and $l=1$, we obtain in Theorem \ref{MainThm1} an explicit formula for $\Delta_1(t)$ in terms of the multiplier ideal of $\sum_{j=1}^r\{\frac{km_j}{d}\}C_j$, where we write $C_j$ instead of $D_j$ when $D$ is a curve $C$. This is the most important result of the article. 

There is another result in the present article, Theorem \ref{Thm4.4}, which discusses the relation between the Hodge spectrum and the monodromy zeta function of a homogeneous singularity. In fact, this result can be realized directly from \cite[Proposition 4.3]{Bu4} and Proposition \ref{dimGr}. 



\section{Multiplier ideals and Hodge spectrum}\label{Section2}
\subsection{Multiplier ideals}
Let $X$ be a smooth complex algebraic variety and let $D=\{D_1,\dots, D_r\}$ be a family of closed subschemes of $X$. A {\it log-resolution} of the family $D$ is a proper birational morphism $\pi: Y\to X$, where $Y$ is a smooth complex algebraic variety, such that the exceptional set $\Ex(\pi):=\{y\in Y \mid \text{$\pi$ is not biregular at $y$}\}$, the support $\Supp(\det\Jac_{\pi})$ of the determinant of the Jacobian of $\pi$, the preimages $\pi^{-1}(D_j)$, $1\leq j\leq r$, and the union $\Ex(\pi)\cup \Supp(\det\Jac_{\pi})\cup \bigcup_{j=1}^r\pi^{-1}(D_j)$ are simple normal crossing divisors. The existence of such a log-resolution is proved by Hironaka. Let $K_X$ (resp. $K_Y$) denote the canonical divisor of $X$ (resp. $Y$). Then $K_{Y/X}:=K_Y-\pi^*K_X$ is the divisor defined by $\det\Jac_{\pi}$, which is known as the canonical divisor of $\pi$. For any $\alpha=(\alpha_1,\dots,\alpha_r)\in \mathbb Q_{>0}^r$, we set
\begin{align}
\mathcal J(X,\alpha D):=\pi_*\mathcal O_Y(K_{Y/X}-\lfloor{\pi^*(\alpha D)}\rfloor),
\end{align}
where $\alpha D:=\sum_{j=1}^r\alpha_j D_j$, and $\lfloor{\pi^*(\alpha D)}\rfloor$ is the round-down of the coefficients of the irreducible components of the divisor $\pi^*(\alpha D)$. It is obvious that $\mathcal J(X,\alpha D)$ is a sheaf of ideals on $X$, which is an ideal of $\pi_*\mathcal O_Y(K_Y)=\mathcal O_X$. 

\begin{theorem}[Lazarsfeld \cite{La1}]
For any $\alpha\in \mathbb Q_{>0}^r$, the sheaf of ideals $\mathcal J(X,\alpha D)$ is independent of the choice of $\pi$, and $R^i\mathcal J(X,\alpha D)=0$ for $i\geq 1$. The sheaf of ideals $\mathcal J(X,\alpha D)$ is called the {\rm multiplier ideal} of $\alpha D$.
\end{theorem}

For instance, when $X=\mathbb C^n$ and $D$ is defined by a monomial ideal $I$, by Howald \cite{Ho}, the multiplier ideal $\mathcal J(X,\alpha D)$ is the ideal generated by $x_1^{\gamma_1}\cdots x_n^{\gamma_n}$ for all $(r_1,\dots,r_n)\in \mathbb N^n$ such that $(\gamma_1+1,\dots,\gamma_n+1)$ is in the interior $\Int(\alpha \Gamma(I))$ of $\alpha \Gamma(I)$, where $\Gamma(I)$ is the Newton polyhedron of $I$. 

Now let $D$ be a closed subscheme of $X$. A {\it jumping number} of $D$ in $X$ is a number $\alpha\in\mathbb Q_{>0}$ such that $\mathcal J(X,\alpha D)\not= \mathcal J(X,(\alpha-\varepsilon)D)$ for all $\varepsilon >0$. The {\it log canonical threshold} $\lct(X,D)$ of $(X,D)$ is the smallest jumping number of $D$ in $X$. In \cite{Mus}, Mustata proves a formula of $\lct(X,D)$ in terms of the discrepancies and multiplicities of a log-resolution of $D$. To determine how a singular point affects a jumping number, Budur \cite{Bu1} introduces {\it inner jumping multiplicities}. By definition, the inner jumping multiplicity $m_{\alpha,\p}(D)$ of $\alpha$ at a closed point $\p\in D$ is the dimension of the complex vector space 
$$\mathcal K_{\p}(X,\alpha D):=\mathcal J(X,(\alpha-\varepsilon)D)/\mathcal J(X,(\alpha-\varepsilon)D+\delta\{\p\}),$$ 
for $0<\varepsilon\ll \delta \ll 1$. If $m_{\alpha,\p}(D)\not=0$, the number $\alpha$ is called an {\it inner jumping number} of $(X,D)$ at $\p$. It is proved by Budur in \cite[Proposition 2.8]{Bu1} that if $\alpha$ is an inner jumping number of $(X,D)$ at $\p$, for some $\p\in D$, then $\alpha$ is a jumping number of $(X,D)$. Furthermore, Budur can provide an explicit formula computing the number $m_{\alpha,\p}(D)$, which we recall as follows. Let $\pi: Y\to X$ be a log-resolution of the family $\{D,\{\p\}\}$, with $E=\pi^*(D)=\sum_{i\in A}N_iE_i$, $E_i$ irreducible components, and, for $d\in\mathbb N_{>0}$, let $J_{d,\p}:=\{i\in A \mid N_i\not=0,\ d|N_i,\ \pi(E_i)=\p\}$ and $E_{d,\p}:=\bigcup_{i\in J_{d,\p}}E_i$. 

\begin{proposition}[Budur \cite{Bu1}, Proposition 2.7]
Assume $\alpha=\frac k d$, with $k$ and $d$ coprime positive integers. Then $m_{\alpha,\p}(D)=\chi(Y,\mathcal O_{E_{d,\p}}(K_{Y/X}-\lfloor{(1-\varepsilon)\alpha\pi^*D}\rfloor))$, where $\chi$ is the sheaf Euler characteristic and $0<\varepsilon\ll 1$.
\end{proposition}


\subsection{Hodge spectrum}
Let $X$ be a smooth complex variety of pure dimension $n$, let $f$ be a regular function on $X$ with zero locus $D\not=\emptyset$, and let $\p$ be a closed point in $D_{\red}$. Fixing a smooth metric on $X$ we may define a closed ball $B(\p,\varepsilon)$ around $\p$ in $X$ and a punctured closed disc $D_{\delta}^*$ around the origin of $\mathbb A_{\mathbb C}^1$. It is well known (cf. \cite{Mil}) that, for $0<\delta \ll\varepsilon \ll 1$, the map 
$$f: B(\p,\varepsilon)\cap f^{-1}(D_{\delta}^*)\to D_{\delta}^*$$ 
is a smooth locally trivial fibration, called Milnor fibration, whose diffeomorphism type is independent of such $\varepsilon$ and $\delta$. Denote the Milnor fiber $B(\p,\varepsilon)\cap f^{-1}(\delta)$ by $M_{\p}$, the geometric monodromy $M_{\p}\to M_{\p}$ and its cohomology level $H^*(M_{\p},\mathbb C)\to H^*(M_{\p},\mathbb C)$ by the same symbol $T$. 

Let $\HS_{\mathbb C}^{\mon}$ be the abelian category of complex mixed Hodge structures endowed with an automorphism of finite order. For an object $(H,T_H)$ of $\HS_{\mathbb C}^{\mon}$, one defines its Hodge spectrum as a fractional Laurent polynomial 
$$\hsp(H,T_H):=\sum_{\alpha\in \mathbb Q}n_{\alpha}t^{\alpha},$$
where $n_{\alpha}:=\dim_{\mathbb C}Gr_F^{\lfloor{\alpha}\rfloor}H_{e^{2\pi i\alpha}}$, $H_{e^{2\pi i\alpha}}$ is the eigenspace of $T_H$ with respect to the eigenvalue $e^{2\pi i\alpha}$, and $F$ is the Hodge filtration. By \cite{St} and \cite{Sa}, for any $l$, $H^l(M_{\p},\mathbb C)$ carries a canonical mixed Hodge structure, which is compatible with the semisimple part $T_s$ of $T$ so that $(H^l(M_{\p},\mathbb C),T_s)$ is an object of $\HS_{\mathbb C}^{\mon}$. As in \cite[Section 4.3]{DL} and \cite[Section 3]{Bu1}, we set
$$\hsp'(f,\p):=\sum_{j\in \mathbb Z}(-1)^j\hsp(\widetilde{H}^{n-1+j}(M_{\p},\mathbb C),T_s),$$
where we use the reduced cohomology $\widetilde H$ to present the vanishing cycle sheaf cohomology, since $\widetilde{H}^l(M_{\p},\mathbb C)_{e^{2\pi i\alpha}}=H^l(M_{\p},\mathbb C)_{e^{2\pi i\alpha}}$ if $l\not=0$ or $\alpha\not\in \mathbb Z$, $\widetilde{H}^0(M_{\p},\mathbb C)_1=\coker(H^0(\ast,\mathbb C)\to H^0(M_{\p},\mathbb C)_1)$ (cf. also \cite[Section 5.1]{BS1}). Then the Hodge spectrum of $f$ at $\p$, denoted by $\Sp(f,\p)$, is the following 
$$\Sp(f,\p)=t^n\iota(\hsp'(f,\p)),$$
where $\iota$ is given by $\iota(t^{\alpha})=t^{-\alpha}$. Writing $\Sp(f,\p)=\sum_{\alpha\in \mathbb Q}n_{\alpha,\p}(f)t^{\alpha}$ one calls the coefficients $n_{\alpha,\p}(f)$ the {\it spectrum multiplicities} of $f$ at $\p$. By \cite[Proposition 5.2]{BS1}, $n_{\alpha,\p}(f)=0$ if $\alpha$ is a rational number with $\alpha\leq 0$ or $\alpha\geq n$. Moreover, it implies from \cite[Corollary 2.3]{Bu4} that, for $\alpha\in (0,n)\cap \mathbb Q$, 
\begin{align}\label{Hodgemult}
n_{\alpha,\p}(f)=\sum_{j\in \mathbb Z}(-1)^j\dim_{\mathbb C}Gr_F^{\lfloor{n-\alpha}\rfloor}H^{n-1+j}(M_{\p},\mathbb C)_{e^{-2\pi i\alpha}}.
\end{align}
Specially, using \cite[Corollary 4.3.1]{DL} and important computations on multiplier ideals, Budur \cite{Bu1} proved the following result, which provides an effective way to compute $n_{\alpha,\p}(f)$, for $\alpha\in (0,1]\cap \mathbb Q$.

\begin{theorem}[Budur \cite{Bu1}]\label{BuCompare}
Let $X$ be a smooth quasi-projective complex variety, and D an effective integral divisor on $X$. Assume that $\p$ is a closed point of $D_{\red}$ and $f$ is any local equation of $D$ at $\p$. Then, for any $\alpha\in (0,1]\cap \mathbb Q$, $n_{\alpha,\p}(f)=m_{\alpha,\p}(D)$.
\end{theorem}

Remark from Theorem \ref{BuCompare} that, for $\alpha\in (0,1]$, $t^{\alpha}$ appears in $\Sp(f,\p)$ if and only if $\alpha$ is an inner jumping number of $(X,D)$ at $\p$. If $\p$ is an isolated singularity of $D$, Theorem \ref{BuCompare} may be even applied to the previous remark when replacing $X$ by an open neighborhood of $X$ to obtain Varchenko's result \cite{Va} (see Corollary in \cite[Section 1]{Bu1}).


\section{Local systems and Milnor fibers of homegeneous singularities}\label{Section3}
\subsection{Local systems and normal $G$-covers}\label{LSCC}
Let us recall some basic notions of local systems and cyclic covers in \cite{EV} and \cite{Bu3}. A complex {\it local system} $\mathcal V$ on a complex manifold is a locally constant sheaf of finite dimensional complex vector spaces. The rank of a locally constant sheaf is the dimension of a stalk as a complex vector space. As mentioned in Budur \cite{Bu3}, local systems of rank one on a complex manifold $U$ correspond to morphisms of groups $H_1(U)\to \mathbb C^*$. In this correspondence, a local system is called {\it unitary} if it is sent to a morphism of groups $H_1(U)\to S^1=\{\eta\in\mathbb C^*\mid |\eta|=1\}$. The constant sheaf $\mathbb C_U$ and any local system of rank one of finite order are simple examples of unitary local systems. 

Let $X$ be a smooth complex projective variety of dimension $n$, and $f$ a regular function on $X$ with zero divisor $D:=f^{-1}(0)$. Denote $U:=X\setminus D$ and write $D_{\red}=\bigcup_{j=1}^rD_j$, where $D_j$ are distinct irreducible reduced subvarieties of $D$. We may use $D$ as the family $\{D_1,\dots,D_r\}$ by abuse of notation (and the following definition will be in this sense), we write $c_1(\mathscr L)$ for the first Chern class of a line bundle $\mathscr L$ and consider the group
$$\Pic^{\tau}(X,D):=\big\lbrace(\mathscr L,\alpha)\in \Pic(X)\times [0,1)^r \mid c_1(\mathscr L)=\alpha [D]\in H^2(X,\mathbb R) \big\rbrace.$$
with the following operation
\begin{align}\label{lawgroup}
(\mathscr L,\alpha)\cdot (\mathscr L',\alpha'):=(\mathscr L\otimes\mathscr L'\otimes\mathcal O_X(-\lfloor (\alpha+\alpha')D\rfloor)),\{\alpha+\alpha'\}),
\end{align}
where $\alpha [D]:=\sum_{j=1}^r\alpha_j [D_j]$, an $\mathbb R$-linear combination of the cohomology classes $[D_j]$ in $H^2(X,\mathbb R)$, and as above $\alpha D:=\sum_{j=1}^r\alpha_j D_j$, $\lfloor \alpha \rfloor:=(\lfloor\alpha_1\rfloor,\dots,\lfloor\alpha_r\rfloor)$ and $\{\alpha\}:=\alpha-\lfloor \alpha \rfloor$. 

\begin{theorem}[Budur \cite{Bu3}, Theorem 1.2]\label{RH}
There is a canonical isomorphism of groups 
\begin{align*}
RH: \Pic^{\tau}(X,D)\cong \Hom(H_1(U),S^1),
\end{align*}
\end{theorem}

By this, one may identify a unitary local system of rank one on $U$ with an element of $\Pic^{\tau}(X,D)$.

Let $\pi: Y\to X$ be a log-resolution of the family $\{D_1,\dots,D_r\}$, and $E:=Y\setminus \pi^{-1}(U)=\sum_{i\in A}N_iE_i$, with $E_j$ irreducible. We shall use the following two important results.

\begin{proposition}[Budur \cite{Bu3}, Proposition 3.3]\label{PicYX}
The map $\pi^*_{par}: \Pic^{\tau}(X,D)\to \Pic^{\tau}(Y,E)$ which sends $(\mathscr L,\alpha)$ to $(\pi^*\mathscr L-\lfloor{\beta E}\rfloor,\{\beta\})$ with $\beta$ defined by $\pi^*(\alpha D)=\beta E$ is an isomorphism of groups.
\end{proposition}

\begin{theorem}[Budur \cite{Bu2}, Theorem 4.6]\label{key}
Let $\mathcal V$ be a rank one unitary local system on $U$ which corresponds to $(\mathscr L,\alpha)\in \Pic^{\tau}(X,D)$. Then, for all $p, q\in\mathbb N$, we have
\begin{align*}
Gr_F^pH^{p+q}(U,\mathcal V^{\vee})=H^{n-q}(Y,\Omega_Y^p(\log E)^{\vee}\otimes \omega_Y\otimes \pi^*\mathscr L \otimes\mathcal O_Y(-\lfloor \pi^*(\alpha D) \rfloor))^{\vee}.
\end{align*}
In particular, 
$$Gr_F^0H^q(U,\mathcal V^{\vee})=H^{n-q}(X,\omega_X\otimes \mathscr L \otimes\mathcal J(X,\alpha D))^{\vee}.$$
\end{theorem}

Let $G$ be a finite group. By \cite[Corollary 1.10]{Bu3}, the dual group $G^*=\Hom(G,\mathbb C^*)$ of $G$ gives rise to a normal $G$-cover of $X$ unramified above $U$. Namely, the normal $G$-cover of $X$ is the morphism
$$\phi: \widetilde X=\Spec_{\mathcal O_X}\left(\bigoplus_{\eta\in G^*}\mathscr L_{\eta}^{-1}\right)\to X$$
induced by the $\mathcal O_X$-module structural morphisms $\mathcal O_X\to \mathscr L_{\eta}$, for all $\eta\in G^*$, where we identify $G^*$ with the subgroup $\{(L_{\eta},\alpha_{\eta}) \mid \eta\in G^*\}$ of $\Pic^{\tau}(X,D)$. The group $G$ acts on $\mathscr L_{\eta}^{-1}$ via the character $\eta$, hence acts on the $\mathcal O_X$-module sheaf $\phi_*\mathcal O_{\widetilde X}$. By \cite[Corollary 1.11]{Bu3}, $\phi_*\mathcal O_{\widetilde X}$ admits an eigensheaf decomposition 
\begin{align}\label{decomp1}
\phi_*\mathcal O_{\widetilde X}=\bigoplus_{\eta\in G^*}\mathscr L_{\eta}^{-1},
\end{align}
where the eigensheaf $\mathscr L_{\eta}^{-1}$ is with respect to the eigenvalue $\eta$ of the action of $G$ on $\phi_*\mathcal O_{\widetilde X}$. 

Now we consider the log-resolution $\pi$. By Proposition \ref{PicYX}, since $\{(\mathscr L_{\eta},\alpha_{\eta}) \mid \eta\in G^*\}$ is a finite subgroup of $\Pic^{\tau}(X,D)$, $\{(\pi^*\mathscr L_{\eta}-\lfloor{\beta_{\eta} E}\rfloor,\beta_{\eta}) \mid \eta\in G^*\}$, with $\beta_{\eta}$ defined by $\pi^*(\alpha_{\eta} D)=\beta_{\eta} E$, is a finite subgroup of $\Pic^{\tau}(Y,E)$. By the same way as previous we can construct the corresponding normal $G$-cover of $Y$ unramified above $\pi^{-1}(U)\cong U$ as follows
$$\rho: \widetilde Y=\Spec_{\mathcal O_Y}\left(\bigoplus_{\eta\in G^*}\pi^*\mathscr L_{\eta}^{-1}\otimes \mathcal O_Y\left(\lfloor{\beta_{\eta}E}\rfloor\right)\right)\to Y,$$
where the group $G$ of acts on $\widetilde Y$ and on $\rho_*\mathcal O_{\widetilde Y}$. Moreover, similarly as (\ref{decomp1}), we have

\begin{proposition}[Budur \cite{Bu3}, Corollary 1.12]\label{decomp2}
There is an eigensheaf decomposition 
\begin{align*}
\rho_*\mathcal O_{\widetilde Y}=\bigoplus_{\eta\in G^*}\pi^*\mathscr L_{\eta}^{-1}\otimes \mathcal O_Y\left(\lfloor{\beta_{\eta}E}\rfloor\right),
\end{align*}
the eigensheaf $\pi^*\mathscr L_{\eta}^{-1}\otimes \mathcal O_Y\left(\lfloor{\beta_{\eta}E}\rfloor\right)$ is with respect to the eigenvalue $\eta$ of the action of $G$ on $\rho_*\mathcal O_{\widetilde Y}$. 
\end{proposition}


\subsection{Milnor fibers of homegeneous singularity}\label{MF}
Let $f(x_0,\dots,x_n)\in \mathbb C[x_0,\dots,x_n]$ be a homogeneous polynomial of degree $d$. We shall take $f$ into two closely interactive entities, a Milnor fiber at the origin of $\mathbb C^{n+1}$ and a complex projective hypersurface of $\mathbb P^n$. By \cite[Lemma 9.4]{Mil}, the Minor fiber $M$ of $f$ at the origin of $\mathbb C^{n+1}$ is diffeomorphic to $\{(x_0,\dots,x_n)\in \mathbb C^{n+1}\mid f(x_0,\dots,x_n)=1\}$. The geometric monodromy $M\to M$ is given by multiplication of elements of $M$ by $e^{\frac{2\pi i} d}$, which induces an endomorphism $T$ of the complex vector space $H^*(M,\mathbb C)$. 

Following \cite[Section 4]{Bu4}, we consider the smooth complex projective variety $X=\mathbb P^n$ and the closed subscheme $D$ of $X$ defined by the zero locus of $f$. Put $U:=X\setminus D$. Since the action of $\mathbb Z/d \mathbb Z$ on $M$ is free, we have a natural isomorphism $M/(\mathbb Z/d \mathbb Z)\cong U$. Denote by $\sigma$ the quotient map $M\to U$, which is the cyclic cover of degree $d$ of $U$. Then there is an eigensheaf decomposition of the $\mathcal O_U$-module sheaf $\sigma_*\mathbb C_M$ as follows
$$\sigma_*\mathbb C_M=\bigoplus_{k=0}^{d-1}\mathcal V_k,$$ 
where $\mathcal V_k$ is the unitary local system on $U$ given by the eigensheaf of $T$ with respect to the eigenvalue $e^{-\frac{2\pi ik} d}$. This implies that 
$$H^l(U,\sigma_*\mathbb C_M)=\bigoplus_{k=0}^{d-1}H^l(U,\mathcal V_k).$$
Let us consider the Leray spectral sequence
$$E_2^{p,q}=H^q(U,R^p\sigma_*\mathbb C_M)\Rightarrow H^{p+q}(M,\mathbb C_M).$$
Since $\sigma$ is a finite morphism of schemes, $R^p\sigma_*\mathbb C_M=0$ for all $p\geq 1$, hence, by this spectral sequence, we have $H^l(U,\sigma_*\mathbb C_M)= H^l(M,\mathbb C_M)= H^l(M,\mathbb C)$, for $l\in \mathbb N$. This implies the following important lemma (cf. \cite[Section 4]{Bu4}). 

\begin{lemma}[Budur \cite{Bu4}]\label{eigenvector}
The complex vector space $H^l(U,\mathcal V_k)$ if nontrivial is the eigenspace of the monodromy action $T$ on $H^l(M,\mathbb C)$ with respect to the eigenvalue $e^{-\frac{2\pi ik} d}$, that is, 
$$H^l(M,\mathbb C)_{e^{-\frac{2\pi ik} d}}=H^l(U,\mathcal V_k),$$ 
for $0\leq k\leq d-1$ and $l\geq 0$.
\end{lemma}

In fact, there are two commuting monodromy actions on $H^l(M,\mathbb C)$, where the endomorphism $T$ is the first one. The second one is, for each $k$, the monodromy of $\mathcal V_k$ around a generic point of $D_j$, which, by \cite[Lemma 4.1]{Bu4}, is given by multiplication by $e^{\frac{2\pi ikm_j} d}$. Together with \cite[Proposition 3.3]{Bu3}, it proves the following important lemma. 

\begin{lemma}[Budur \cite{Bu4}, Lemma 4.2]\label{keylemma1}
Assume $D=\sum_{j=1}^rm_jD_j$, with $D_j$ irreducible of degree $d_j$. Then the element in $\Pic^{\tau}(X,D)$ corresponding via the isomorphism $RH$ in Theorem \ref{RH} to the unitary local system $\mathcal V_k$ is $(\mathcal O_{\mathbb P^n}(\sum_{j=1}^r\{\frac{km_j}{d}\}d_j), (\{\frac{km_1}{d}\},\dots,\{\frac{km_r}{d}\}))$. 
\end{lemma}

Notice that $\sum_{j=1}^r\{\frac{km_j}{d}\}d_j$ is an integer, because, if for every $1\leq j\leq r$ we write $km_j=dn_j+s_j$, with $n_j, s_j\in \mathbb N$, $0\leq s_j<d$, we have
\begin{align*}
\sum_{j=1}^r\left\{\frac{km_j}{d}\right\}d_j=\sum_{j=1}^r\frac{s_jd_j}{d}=\sum_{j=1}^r\frac{km_jd_j-dn_jd_j}{d}=k-\sum_{j=1}^rn_jd_j.
\end{align*}

Fix a log-resolution $\pi: Y\to \mathbb P^n$ of the family of closed subschemes $\{D_1,\dots, D_r\}$ of $\mathbb P^n$, and, as previous, denote $E=\pi^*(\bigcup_{j=1}^rD_j)=\sum_{j\in A}N_jE_j$, with $E_j$ irreducible components of $\pi^{-1}(D)$. Let 
\begin{align}\label{Lk}
\mathscr L^{(k)}:=\pi^*\mathcal O_{\mathbb P^n}\left(\sum_{j=1}^r\left\{\frac{km_j}{d}\right\}d_j\right) \otimes \mathcal O_Y\left(-\left\lfloor{\sum_{j=1}^r\left\{\frac{km_j}{d}\right\}}\pi^*D_j\right\rfloor\right).
\end{align}
Denote by $B$ the set of integers $k$ such that $0\leq k\leq d-1$ and $d$ divides $km_j$ for all $1\leq j\leq r$, and by $\overline{B}$ the complement of $B$ in $[0,d-1]\cap\mathbb Z$. 

\begin{remark}\label{Remark1}
If $k$ is in $B$, then $\mathscr L^{(k)}=\mathcal O_Y$. Furthermore, if $k$ is in $B$ and $k\not=0$, so is $d-k$; if $k$ and $k'$ are in $B$, so is either $k+k'$ or $k+k'-d$; hence we can consider $B$ as a subgroup of $\mathbb Z/d\mathbb Z$. Let $m=\gcd(m_1,\dots,m_r)$, and $u_j\in \mathbb N_{>0}$ with $m_j=mu_j$ for $1\leq j\leq r$. Then $k\in B$ if and only if $0\leq k\leq d-1$ and $ku_s$ is divisible by $\sum_{j=1}^rd_ju_j$ for any $1\leq s\leq r$. Since $u_1,\dots,u_r$ are coprime, the latter means that $k$ is divisible by $\sum_{j=1}^rd_ju_j$, hence the cardinal $|B|$ of $B$ equals $m$. 
\end{remark}

For simplicity of notation, from now on, if $\mathscr A$ is a sheaf on $\mathbb P^n$, and $l\in\mathbb Z$, we shall write $\mathscr A(l)$ in stead of $\mathscr A\otimes \mathcal O_{\mathbb P^n}(l)$. 


\begin{proposition}\label{dimGr}
With the notation as in Lemma \ref{keylemma1} we have
\begin{itemize}
\item[(i)] $\dim_{\mathbb C}Gr_F^pH^{p+q}(U,\mathcal V_k)=\dim_{\mathbb C}H^q(Y,\Omega_Y^p(\log E))$, for $k\in B$;
\item[(ii)] $\dim_{\mathbb C}Gr_F^pH^{p+q}(U,\mathcal V_{d-k})=\dim_{\mathbb C}H^q(Y,\Omega_Y^p(\log E)\otimes {\mathscr L^{(k)}}^{-1})$, for $k\in\overline{B}$.  
\end{itemize}
In particular, for $k\in\overline{B}$,

$\dim_{\mathbb C} Gr_F^0H^q(U,\mathcal V_{d-k})=\dim_{\mathbb C}H^{n-q}\left(\mathbb P^n,\mathcal J\left(\mathbb P^n,\sum_{j=1}^r\left\{\frac{km_j}{d}\right\}D_j\right)\left(\sum_{j=1}^r\left\{\frac{km_j}{d}\right\}d_j-n-1\right)\right).$
\end{proposition}

\begin{proof}
Due to the group law (\ref{lawgroup}) of $\Pic^{\tau}(X,D)$ and definition of $\mathcal V_k$, it is obvious that $\mathcal V_k=\mathcal V_k^{\vee}=\mathcal V_0$ for $k\in B$, and that $\mathcal V_{d-k}=\mathcal V_k^{\vee}$ for $k\in\overline{B}$. Then, by Lemma \ref{keylemma1} and Theorem \ref{key}, we have
$$Gr_F^pH^{p+q}(U,\mathcal V_k)=H^{n-q}(Y,\Omega_Y^p(\log E)^{\vee}\otimes \omega_Y)^{\vee}$$ 
for $k\in B$, and 
\begin{align*}
Gr_F^pH^{p+q}(U,\mathcal V_{d-k})&=H^{n-q}(Y,\Omega_Y^p(\log E)^{\vee}\otimes \omega_Y\otimes \mathscr L^{(k)})^{\vee}\\
&=H^{n-q}(Y,(\Omega_Y^p(\log E) \otimes {\mathscr L^{(k)}}^{-1})^{\vee}\otimes \omega_Y)^{\vee},
\end{align*}
for $k\in\overline{B}$. Applying the Serre duality we obtain (i) and (ii).

For the rest statement, we again apply Lemma \ref{keylemma1} and the particular case in Theorem \ref{key}, together with the definition of multiplier ideal.
\end{proof}

Denote $\mathscr L_{\red}^{(k)}:=\pi^*\mathcal O_{\mathbb P^n}(k)\otimes \mathcal O_Y(-\lfloor\frac{k}{d}E\rfloor)$, for $0\leq k\leq d-1$. 

\begin{corollary}\label{Cor3.9}
With the notation as in Lemma \ref{keylemma1} and $D$ being reduced, for $1\leq k\leq d$,
\begin{itemize}
\item[(i)] $\dim_{\mathbb C}Gr_F^pH^{p+q}(U,\mathcal V_{d-k})=\dim_{\mathbb C}H^q(Y,\Omega_Y^p(\log E)\otimes {\mathscr L_{\red}^{(k)}}^{-1})$;
\item[(ii)] $\dim_{\mathbb C} Gr_F^0H^q(U,\mathcal V_{d-k})=\dim_{\mathbb C}H^{n-q}(\mathbb P^n,\mathcal J(\mathbb P^n,\frac{k}{d}D)(k-n-1))$.
\end{itemize}
\end{corollary}

\begin{proof}
Applying Proposition \ref{dimGr} to the special case $m_1=\cdots=m_r=1$ we obtain the statements. Note that, in this case, $B=\{0\}$ and $\overline{B}=\{1,\dots,d-1\}$.
\end{proof}

\begin{lemma}\label{dimHi}
With the notation as in Lemma \ref{keylemma1}, and by observation $\mathscr L^{(d)}=\mathscr L^{(0)}$, we have
\begin{itemize}
\item[(i)] $\dim_{\mathbb C}H^1(U,\mathcal V_k)=r-1$, if $n=2$ and $k\in B$;
\item[(ii)] $\dim_{\mathbb C}H^l(U,\mathcal V_{d-k})=\sum_{p\geq 0}\dim_{\mathbb C}H^{j-p}(Y,\Omega_Y^p(\log E)\otimes {\mathscr L^{(k)}}^{-1})$, if $j\geq 0$ and $1\leq k\leq d$.
\end{itemize}
\end{lemma}

\begin{proof}
By Proposition \ref{dimGr} (i), if $k\in B$, we have $\dim_{\mathbb C}Gr_F^pH^{p+q}(U,\mathcal V_k)=\dim_{\mathbb C}H^q(Y,\Omega_Y^p(\log E))$, thus 
$$\dim_{\mathbb C}H^1(U,\mathcal V_0)=\dim_{\mathbb C}H^1(Y,\mathcal O_Y)+\dim_{\mathbb C}H^0(Y,\Omega_Y^1(\log E)).$$
Assume that $n=2$. It is a fact that $\dim_{\mathbb C}H^1(Y,\mathcal O_Y)=0$, because $Y$ is birationally equivalent to $\mathbb P^2$, and that $\dim_{\mathbb C}H^0(Y,\Omega_Y^1(\log E))=r-1$, due to the proof of Th\'eor\`eme 6 in \cite{Esn}. This proves (i). 

The statement (ii) of this lemma is a consequence of Proposition \ref{dimGr} (ii). 
\end{proof}

\section{Monodromy characteristic polynomials and zeta function}
As in Section \ref{MF}, we shall work with a homogeneous polynomial $f(x_0,\dots,x_n)\in\mathbb C[x_0,\dots,x_n]$ of degree $d$. By considering its germ at the origin of $\mathbb C^{n+1}$ we study some singularity invariants, including the characteristic polynomials and the zeta function of the monodromy.

\subsection{Characteristic polynomials}
Recall that the Milnor fiber $M$ of the singularity $f(x_0,\dots,x_n)$ at the origin of $\mathbb C^{n+1}$ is diffeomorphic to $\{(x_0,\dots,x_n)\in\mathbb C^{n+1} \mid f(x_0,\dots,x_n)=1\}$, and the monodromy $T$ is induced by $e^{\frac{2\pi i} d}\cdot(x_0,\dots,x_n)=(e^{\frac{2\pi i} d} x_0,\dots,e^{\frac{2\pi i} d} x_n)$. By definition, the (monodromy) characteristic polynomial $\Delta_l(t)$ of the endomorphism $T|_{H^l(M,\mathbb C)}$ of $H^l(M,\mathbb C)$ is the monic polynomial 
$$\Delta_l(t)=\det(t\Id-T|_{H^l(M,\mathbb C)}).$$
Assume that 
$$f(x_0,\dots,x_n)=\prod_{j=1}^rf_j(x_0,\dots,x_n)^{m_j},$$ 
where $f_j(x_0,\dots,x_n)$ are distinct irreducible homogeneous polynomials of degree $d_j$, $1\leq j\leq r$. As above, we denote by $D_j$ the complex projective plane curve $\{(x_0:\dots:x_n)\in\mathbb P^n \mid f_j(x_0,\dots,x_n)=0\}$, for $1\leq j\leq r$.

Fix a log-resolution $\pi: Y\to \mathbb P^n$ of the family $D=\{D_1,\dots, D_r\}$, with normal crossing divisor $E=\pi^{-1}(\bigcup_{j=1}^rD_j)$. As mentioned in Section \ref{Section3}, there is an isomorphism $M/(\mathbb Z/d\mathbb Z)\cong U=\mathbb P^n\setminus D$ so that the  canonical projection $\sigma: M\to U$ induces an eigensheaf decomposition $\sigma_*\mathbb C_M=\bigoplus_{k=0}^{d-1}\mathcal V_k$, where $\mathcal V_k$ are the unitary local systems on $U$ given in Lemma \ref{keylemma1}. By Lemma \ref{eigenvector}, for $1\leq k\leq d$ and $l\in  \mathbb N$, the vector space $H^l(U,\mathcal V_{d-k})$ if nontrivial is the eigenspace of $T|_{H^l(M,\mathbb C)}$ with respect to the eigenvalue $e^{\frac{2\pi ik} d}$. This together with Lemma \ref{dimHi} and Remark \ref{Remark1} proves the following lemma. 

\begin{lemma}\label{Alexander1}
Let $\Delta_l(t)$ be the characteristic polynomial of the endomorphism $T|_{H^l(M,\mathbb C)}$ of $H^l(M,\mathbb C)$. Then, with the previous notation and $l\in\mathbb N$, one has 
$$\triangle_l(t)=\prod_{k=0}^{d-1}(t-e^{\frac{2\pi ik} d})^{h_l^{(k)}},$$
where, 
$$h_l^{(k)}:=\dim_{\mathbb C}H^l(U,\mathcal V_{d-k})=\sum_{p+q=l}h^q(\Omega_Y^p(\log E)\otimes {\mathscr L^{(k)}}^{-1}),$$ 
with $h^q(\Omega_Y^p(\log E)\otimes {\mathscr L^{(k)}}^{-1})=\dim_{\mathbb C}H^q(Y,\Omega_Y^p(\log E)\otimes {\mathscr L^{(k)}}^{-1})$, and
$$\mathscr L^{(k)}=\pi^*\mathcal O_{\mathbb P^n}\left(\sum_{j=1}^r\left\{\frac{km_j}{d}\right\}d_j\right) \otimes \mathcal O_Y\left(-\left\lfloor{\sum_{j=1}^r\left\{\frac{km_j}{d}\right\}}\pi^*D_j\right\rfloor\right).$$ 
\end{lemma}

As above, we denote by $B$ the set of $k$ in $\mathbb Z$ such that $0\leq k\leq d-1$ and $d$ divides $km_j$ for all $1\leq j\leq r$, by $\overline{B}$ the complement of $B$ in $[0,d-1]\cap\mathbb Z$, and $m=\gcd(m_1,\dots,m_r)$. Due to Remark \ref{Remark1}, $B$ may be considered as a subgroup of $\mathbb Z/d\mathbb Z$. Let $G$ be the quotient group $(\mathbb Z/d\mathbb Z)/B$. For convenience, we shall identify $k\in [0,d-1]\cap \mathbb Z$ with its class in $G$.

\begin{lemma}\label{Alexander2}
With the notation as in Lemma \ref{Alexander1}, one has 
$$\Delta_l(t)=\prod_{k\in G}(t^m-e^{\frac{2\pi ikm} d})^{h_l^{(k)}}$$
for $l\in \mathbb N$. In particular, $\Delta_0(t)=t^m-1$.
\end{lemma}

\begin{proof}
If $k$ and $k'$ belong to the same class in $G$, we have $h_l^{(k)}=h_l^{(k')}$. This together with Lemma \ref{Alexander1} implies the first statement. Since $h^0(\mathcal O_Y)=1$, it remains to check that $h^0({\mathscr L^{(k)}}^{-1})=0$ for $k\in G\setminus\{0\}$. By Lemmas \ref{eigenvector} and \ref{dimHi}, we have 
\begin{align}\label{BCAM19}
\dim_{\mathbb C}H^0(M,\mathbb C)=\sum_{k\in B}h^0({\mathscr L^{(k)}}^{-1})+\sum_{k\in \overline{B}}h^0({\mathscr L^{(k)}}^{-1}).
\end{align} 
It is known that $\dim_{\mathbb C}H^0(M,\mathbb C)=m$ (cf. \cite[Proposition 2.3]{Dimca}). Note that $|B|=m$ (cf. Remark \ref{Remark1}), and that, for $k\in B$, $\mathscr L^{(k)}=\mathcal O_Y$ and $h^0(\mathcal O_Y)=1$. Then (\ref{BCAM19}) is equivalent to $\sum_{k\in \overline{B}}h^0({\mathscr L^{(k)}}^{-1})=0$, which implies that $h^0({\mathscr L^{(k)}}^{-1})=0$ for $k\in\overline{B}$; in particular, $h^0({\mathscr L^{(k)}}^{-1})=0$ for $k\in G\setminus \{0\}$.
\end{proof}

Let us now consider the case where $n=2$. In this case, we shall denote $C$ (resp. $C_j$) instead of $D$ (resp. $D_j$). Then the characteristic polynomial $\Delta_1(t)$ is an important invariant of the homogeneous surface singularity. The following theorem is a main result in the present article.

\begin{theorem}\label{MainThm1}
With the previous notation and $n=2$, one has
$$\Delta_1(t)=(t^m-1)^{r-1}\prod_{k\in G\setminus \{0\}}\left(t^{2m}-2t^m\cos\frac{2km\pi}{d}+1\right)^{\ell_k},$$
where 
$$\ell_k:=\dim_{\mathbb C}H^1\left(\mathbb P^2,\mathcal J\left(\mathbb P^2,\sum_{j=1}^r\left\{\frac{km_j}{d}\right\}C_j\right)\left(\sum_{j=1}^r\left\{\frac{km_j}{d}\right\}d_j-3\right)\right).$$ 
\end{theorem}

\begin{proof}
According to Lemma \ref{Alexander2}, it suffices to prove that 
\begin{align}\label{Eq1}
h^1({\mathscr L^{(k)}}^{-1})=\ell_k
\end{align}
and
\begin{align}\label{Eq2}
h^0(\Omega_Y^1(\log E)\otimes {\mathscr L^{(k)}}^{-1})=\ell_{d-k},
\end{align}
for $k\in G\setminus \{0\}$. The equality (\ref{Eq1}) is a direct corollary of Proposition \ref{dimGr} and Lemma \ref{dimHi}.

To prove (\ref{Eq2}) we consider a common $G$-equivariant desingularization of $\widetilde X$ and $\widetilde Y$, say, $\theta: Z\to \widetilde X$ and $\nu: Z\to \widetilde Y$, in the sense of \cite{Abr}, such that $\pi\circ \rho \circ \nu=\phi\circ \theta=:u$. Here, we use the notation in Section \ref{LSCC} with $X=\mathbb P^2$, and, in particular, we use the the normal $G$-cover of $\mathbb P^2$,
$$\phi: \widetilde X=\Spec_{\mathcal O_{\mathbb P^2}}\left(\bigoplus_{k\in G}\mathcal O_{\mathbb P^2}\Big(-\sum_{j=1}^r\Big\{\frac{km_j}{d}\Big\}d_j\Big)\right)\to \mathbb P^2,$$
and the normal $G$-cover of $Y$,
$$\rho: \widetilde Y=\Spec_{\mathcal O_Y}\left(\bigoplus_{k\in G}{\mathscr L^{(k)}}^{-1}\right)\to Y,$$
where as mentioned previously we identify $k\in [0,d-1]\cap \mathbb Z$ with its class in $G$.

Note that $G^*=\{(\mathcal O_{\mathbb P^2}(\sum_{j=1}^r\{\frac{km_j}{d}\}d_j), (\{\frac{km_1}{d}\},\dots,\{\frac{km_r}{d}\}))\}_{0\leq k\leq d-1}$, which is by Remark \ref{Remark1} a subgroup of order $\frac d m$ of the group $\Pic^{\tau}(\mathbb P^2,C)$. We may choose $Z$ such that $\Delta:=Z\setminus u^{-1}(U)$ is a normal crossing divisor. An analogue of \cite[Corollaire 4]{Esn} shows that, for any $q\in \mathbb N$, 
\begin{equation}\label{Eq4.3}
\begin{aligned}
(\rho\circ\nu)_*\Omega_Z^q(\log\Delta)&\cong \Omega_Y^q(\log E)\otimes (\rho\circ\nu)_*\mathcal O_Z,\\ 
R^p(\rho\circ\nu)_*\Omega_Z^q(\log\Delta)&=0 \quad \text{if} \ p>0
\end{aligned}
\end{equation}
(see also \cite[Lemma 3.22]{EV}). By the Leray spectral sequence
$$E_2^{p,q}=H^q(Y,R^p(\rho\circ\nu)_*\Omega_Z^1(\log\Delta))\Rightarrow H^{p+q}(Z,\Omega_Z^1(\log\Delta))$$
and by (\ref{Eq4.3}), we get, in particular,
\begin{align}\label{Eq4.4}
H^0(Y,\Omega_Y^1(\log E)\otimes (\rho\circ\nu)_*\mathcal O_Z)=H^0(Z,\Omega_Z^1(\log\Delta)).
\end{align}
Since $G^*$ is a finite subgroup of $\Pic^{\tau}(\mathbb P^2,C)$, we deduce from Proposition \ref{decomp2} that 
$$(\rho\circ\nu)_*\mathcal O_Z=\rho_*\mathcal O_{\widetilde Y}=\bigoplus_{k\in G}{\mathscr L^{(k)}}^{-1}.$$
This yields the following decomposition 
\begin{align}\label{Eq4.5}
H^0(Y,\Omega_Y^1(\log E)\otimes (\rho\circ\nu)_*\mathcal O_Z)=\bigoplus_{k\in G}H^0(Y,\Omega_Y^1(\log E)\otimes {\mathscr L^{(k)}}^{-1}).
\end{align}
Note that, due to the proof of Lemma \ref{dimHi}, the direct summand of (\ref{Eq4.5}) corresponding to $k=0$ has complex dimension $r-1$.

Now we compute the dimension of complex vector space on the right hand side of (\ref{Eq4.4}). Similarly as in the proof of Lemma 7 of \cite{Esn}, one may point out that 
\begin{align}\label{Eq4.6}
\dim_{\mathbb C}H^0(Z,\Omega_Z^1(\log\Delta))=\dim_{\mathbb C}H^0(Z,\Omega_Z^1)+(r-1).
\end{align}
On the other hand, by \cite[Corollary 1.13]{Bu3}, we have 
\begin{align}\label{Eq4.7}
H^0(Z,\Omega_Z^1)\cong \bigoplus_{k\in G}H^1\left(\mathbb P^2,\mathcal J\left(\mathbb P^2,\sum_{j=1}^r\left\{\frac{km_j}{d}\right\}C_j\right)\left(\sum_{j=1}^r\left\{\frac{km_j}{d}\right\}d_j-3\right)\right).
\end{align}
In the decomposition (\ref{Eq4.7}), look at the direct summand corresponding to $k=0$. They are nothing but $H^1(\mathbb P^2,\mathcal O_{\mathbb P^2}(-3))=H^1(\mathbb P^2,\omega_{\mathbb P^2})$. By the Serre duality, $\dim_{\mathbb C}H^1(\mathbb P^2,\omega_{\mathbb P^2})=\dim_{\mathbb C}H^1(\mathbb P^2,\mathcal O_{\mathbb P^2})=0$.
Therefore, from (\ref{Eq4.4}), (\ref{Eq4.5}), (\ref{Eq4.6}) and (\ref{Eq4.7}), we get 
\begin{align}\label{Eq4.8}
\sum_{k\in G\setminus\{0\}}h^0(\Omega_Y^1(\log E)\otimes {\mathscr L^{(k)}}^{-1})=\sum_{k\in G\setminus\{0\}}\ell_k.
\end{align}
Repeating the proof of \cite[Proposition 4.6]{LV} and using (\ref{Eq1}) we obtain that 
$$h^0(\Omega_Y^1(\log E)\otimes {\mathscr L^{(k)}}^{-1})\geq \ell_{d-k},$$
for $k\in G\setminus\{0\}$. This together with (\ref{Eq4.8}) means the equality $h^0(\Omega_Y^1(\log E)\otimes {\mathscr L^{(k)}}^{-1})=\ell_{d-k}$, thus (\ref{Eq2}) is proved.
\end{proof}




\subsection{A formula for the monodromy zeta function}\label{Mainresult2}
By definition, the monodromy zeta function the homogeneous singularity $f(x_0,\dots,x_n)$ at the origin $O$ of $\mathbb C^{n+1}$ is the function 
$$\zeta_{f,O}(t)=\prod_{l\geq 0}\det(\Id-tT|_{H^l(M,\mathbb C)})^{(-1)^{l+1}}.$$
This function may be expressed via the polynomials $\Delta_l(t)$ as $\zeta_{f,O}(t)=\prod_{l\geq 0}(t^{\dim_{\mathbb C}H^l(M,\mathbb C)}\Delta_l(\frac 1 t))^{(-1)^{l+1}}$, from which, by Lemma \ref{Alexander2},
\begin{equation}\label{Eq4.10}
\zeta_{f,O}(t)=\prod_{k\in G}\left(1-e^{\frac{2\pi ikm} d}t^m\right)^{\sum_{l\geq 0}(-1)^{l+1}h_l^{(k)}}.
\end{equation}

As explained in \cite{Bu4}, the only numbers $\alpha\in (0,n+1)\cap\mathbb Q$ such that $n_{\alpha,O}(f)$, the coefficients of $t^{\alpha}$ in $\Sp(f,O)$, can be nonzero are of the form $\frac k d+p$, with $k, p\in \mathbb Z$, $1\leq k\leq d$ and $0\leq p\leq n+1$. Then it implies from (\ref{Hodgemult}) and Lemma \ref{eigenvector} that
\begin{equation}\label{Eq4.11}
\begin{aligned}
n_{\frac k d +p,O}(f)=\sum_{j\in \mathbb Z}(-1)^j\dim_{\mathbb C}Gr_F^{n-p}H^{n+j}(U,\mathcal V_k), 
\end{aligned}
\end{equation}
for integers $1\leq k\leq d$ and $0\leq p\leq n+1$, where $\mathcal V_k$ is the local system corresponding to the element $(\mathcal O_{\mathbb P^2}(\sum_{j=1}^r\{\frac{km_j}{d}\}d_j), (\{\frac{km_1}{d}\},\dots,\{\frac{km_r}{d}\}))$ in $\Pic^{\tau}(X,D)$ via the isomorphism $RH$ in Theorem \ref{RH} (cf. Lemma \ref{keylemma1}). Note that $\mathcal V_d=\mathcal V_0$. By Proposition \ref{dimGr} and (\ref{Eq4.11}), we have
\begin{align}\label{Eq4.12}
n_{\frac{d-k}{d}+p,O}(f)=\sum_{j\in \mathbb Z}(-1)^jh^{p+j}(\Omega_Y^{n-p}(\log E)\otimes {\mathscr L^{(k)}}^{-1}), 
\end{align}
for $k\in G$ when $p<n$, and $k\in G\setminus \{0\}$ when $p=n$, where $\mathscr L^{(k)}$ and $h^q(\Omega_Y^p(\log E)\otimes {\mathscr L^{(k)}}^{-1})$ are as in Lemma \ref{Alexander1} (see also (\ref{Lk})). 

\begin{theorem}\label{Thm4.4}
The monodromy zeta function and the Hodge spectrum of the singularity $f$ are fit into a relation as follows
$$\zeta_{f,O}(t)^{(-1)^{n+1}}=\left(1-t^m\right)^{1+\sum_{p=1}^nn_{p,O}(f)}\prod_{k\in G\setminus\{0\}}\left(1-e^{\frac{2\pi ikm} d}t^m\right)^{\sum_{p=0}^nn_{\frac{d-k}{d}+p,O}(f)}.$$
\end{theorem}

\begin{proof}
Recall from Lemma \ref{Alexander1} that $h_l^{(k)}=\sum_{p+q=l}h^q(\Omega_Y^p(\log E)\otimes {\mathscr L^{(k)}}^{-1})$. Since $h^0(\mathcal O_Y)=1$ and $h^q(\mathcal O_Y)=0$ for all $q\geq 1$, the formula (\ref{Eq4.12}) gives
$$
(-1)^{n+1}+(-1)^{n+1}\sum_{p=0}^{n-1}n_{p+1,O}(f)=\sum_{j\in \mathbb Z}(-1)^{n+j+1}h_{n+j}^{(0)}. 
$$
As in the proof of Lemma \ref{Alexander2}, if $k\in G\setminus \{0\}$, then $h^0({\mathscr L^{(k)}}^{-1})=0$, thus by (\ref{Eq4.12}) we have
$$
(-1)^{n+1}\sum_{p=0}^nn_{\frac{d-k}{d}+p,O}(f)=\sum_{j\in \mathbb Z}(-1)^{n+j+1}h_{n+j}^{(k)}. 
$$
Now applying (\ref{Eq4.10}) we obtain the statement of the theorem.
\end{proof}

\begin{remark}
The formula (\ref{Eq4.12}) has the following interesting consequence. Assume that $f(x_0,\dots,x_n)$ is a homogeneous polynomial and has isolated singularity at the origin $O$ of $\mathbb C^{n+1}$. Then the non-trivial characteristic polynomials of the singularity only appear in the degrees $0$ and $n$. This means that $h_l^{(k)}=0$ for all $l\not\in \{0,n\}$ and $0\leq k\leq d-1$. Using the proof of Theorem \ref{Thm4.4}, we obtain the identities  
\begin{align*}
h_n^{(0)}=\sum_{p=1}^nn_{p,O}(f) \quad \text{and}\quad 
h_n^{(k)}&=\sum_{p=0}^nn_{\frac{d-k}{d}+p,O}(f) \ \text{for $1\leq k\leq d-1$},
\end{align*}
which prove the below result. By convention, we may consider the zero space $\{0\}$ as an eigenspace of the monodromy of the singularity with dimension zero.

\begin{corollary}
Let $f(x_0,\dots,x_n)\in \mathbb C[x_0,\dots,x_n]$ be a homogeneous polynomial defining an isolated singularity at the origin $O$ of $\mathbb C^{n+1}$. Then the complex dimension of the eigenspace of the monodromy of the singularity with respect to the eigenvalue $1$ (resp. $e^{\frac{2\pi ik}{d}}$, for $1\leq k\leq d-1$) is $\sum_{p=1}^nn_{p,O}(f)$ (resp. $\sum_{p=0}^nn_{\frac{d-k}{d}+p,O}(f)$, for $1\leq k\leq d-1$).
\end{corollary}
\end{remark}

\begin{ack}
The second and third authors thank Vietnam Institute for Advanced Study in Mathematics and Department of Mathematics - KU Leuven for warm hospitality during their visits. The third author is grateful to Nero Budur for valuable discussions.
\end{ack}


\end{document}